\documentclass[11pt]{article}
\usepackage{epigamath}

\usepackage[notext]{kpfonts}
\usepackage{baskervald}

\setpapertype{A4}

\usepackage[english]{babel}

\usepackage[dvips]{graphicx}     

\title{Irregular Hodge numbers of\\ confluent hypergeometric differential equations}
\titlemark{Irregular Hodge numbers of confluent hypergeometrics}
\author{\vspace{0cm} Claude Sabbah and Jeng-Daw Yu}
\authoraddresses{
\authordata{Claude Sabbah}{\firstname{Claude} \lastname{Sabbah}\\
\institution{CMLS, \'Ecole polytechnique, CNRS, Universit\'e Paris-Saclay,
F--91128 Palaiseau cedex,
France}\\
\email{Claude.Sabbah@polytechnique.edu}\\
\url{http://www.math.polytechnique.fr/perso/sabbah}}\\[1mm]
\authordata{Jeng-Daw Yu}{\firstname{Jeng-Daw} \lastname{Yu}\\
\institution{Department of Mathematics, National Taiwan University, 
Taipei 10617, Taiwan}\\
\email{jdyu@ntu.edu.tw}\\
\url{http://homepage.ntu.edu.tw/~jdyu/}}
}
\authormark{C. Sabbah and J.-D. Yu}
\date{\vspace{-5ex}} 
\journal{\'Epijournal de G\'eom\'etrie Alg\'ebrique} 
\acceptation{Received by the Editors on December 20, 2018, and in final form
on March 20, 2019. \\ Accepted on May 1, 2019.}

\newcommand{\articlereference}{Volume 3 (2019), Article Nr. 7}

\acknowledgement{This research was partly supported by the MoST/CNRS grants 106-2911-I-002-539/TWN PRC 1632. \\ J.-D.Y.\ was partially supported by the TIMS and the MoST of Taiwan.}

\usepackage[all]{xy}

\let\mathcal\mathscr

\newtheorem{assumption}{Assumption}

\newenvironment{theoreme*}{\paragraph{Theorem.}\em}{\rm}
\newenvironment{remarques*}{\paragraph{Remarks.}}{}

\newtheorem{lemme}{Lemma}

\newenvironment{enumeratei}{\begin{enumerate}[\rm (i)]}{\end{enumerate}}
\newenvironment{enumeratea}{\begin{enumerate}[\rm (a)]}{\end{enumerate}}

\def\cD{\mathcal{D}}
\def\cH{\mathcal{H}}
\def\cM{\mathcal{M}}
\def\cO{\mathcal{O}}
\def\cT{\mathcal{T}}

\def\CC{\mathbb{C}}
\def\NN{\mathbb{N}}
\def\PP{\mathbb{P}}

\def\RR{\mathbb{R}}
\def\ZZ{\mathbb{Z}}

\RequirePackage[mathscr]{eucal}
\def\ccE{\EuScript{E}}
\def\ccH{\EuScript{H}}
\def\ccL{\EuScript{L}}
\def\ccM{\EuScript{M}}
\def\ccN{\EuScript{N}}

\def\alphag{\boldsymbol{\alpha}}
\def\betag{\boldsymbol{\beta}}

\newcommand{\bbullet}{{\scriptscriptstyle\bullet}}
\newcommand{\cbbullet}{{\raisebox{1pt}{$\bbullet$}}}
\newcommand{\isom}{\stackrel{\sim}{\longrightarrow}}

\def\cf{see\kern.3em}
\def\eg{e.g.\kern.3em}
\def\ie{i.e.,\ }
\def\resp{\text{resp.}\kern.3em}

\let\leq\leqslant
\let\geq\geqslant
\let\wh\widehat

\let\ov\overline

\newcommand{\sfi}{\mathsf{i}}
\DeclareMathOperator{\twopii}{2\pi\sfi}

\newcommand{\intt}{{\mathrm{int}}}
\newcommand{\irr}{{\mathrm{irr}}}
\newcommand{\regg}{{\mathrm{reg}}}
\newcommand{\rd}{{\mathrm{d}}}

\let\map f
\let\mero\varphi

\DeclareMathOperator{\gr}{gr}
\DeclareMathOperator{\rk}{rk}
\DeclareMathOperator{\IrrMHM}{\mathsf{IrrMHM}}
\DeclareMathOperator{\MTM}{\mathsf{MTM}}

\def\to{\mathchoice{\longrightarrow}{\rightarrow}{\rightarrow}{\rightarrow}}
\def\mto{\mathchoice{\longmapsto}{\mapsto}{\mapsto}{\mapsto}}
\def\hto{\mathrel{\lhook\joinrel\to}}

\def\To#1{\mathchoice{\xrightarrow{\textstyle\kern4pt#1\kern3pt}}{\stackrel{#1}{\longrightarrow}}{}{}}

\let\oldbigoplus\bigoplus
\renewcommand{\bigoplus}{\mathop{\textstyle\oldbigoplus}\displaylimits}
\let\oldbigotimes\bigotimes
\renewcommand{\bigotimes}{\mathop{\textstyle\oldbigotimes}\displaylimits}
\let\oldbigwedge\bigwedge
\renewcommand{\bigwedge}{\mathop{\textstyle\oldbigwedge}\displaylimits}
\let\oldbigcap\bigcap
\renewcommand{\bigcap}{\mathop{\textstyle\oldbigcap}\displaylimits}
\let\oldbigcup\bigcup
\renewcommand{\bigcup}{\mathop{\textstyle\oldbigcup}\displaylimits}
\let\oldprod\prod
\renewcommand{\prod}{\mathop{\textstyle\oldprod}\displaylimits}
\let\oldbigsqcup\bigsqcup
\renewcommand{\bigsqcup}{\mathop{\textstyle\oldbigsqcup}\displaylimits}

\newcommand{\RedefinitSymbole}[1]{%
\expandafter\let\csname old\string#1\endcsname=#1
\let#1=\relax
\newcommand{#1}{\csname old\string#1\endcsname\,}%
}
\RedefinitSymbole{\forall} \RedefinitSymbole{\exists}

\newcommand{\Afu}{\mathbb{A}^{\!1}}

\newcommand{\Clt}{\CC[t]\langle\partial_t\rangle}
\newcommand{\Cltm}{\CC[t,t^{-1}]\langle\partial_t\rangle}
\newcommand{\Cltau}{\CC[\tau]\langle\partial_\tau\rangle}
\newcommand{\Clv}{\CC[v]\langle\partial_v\rangle}
\newcommand{\Clvm}{\CC[v,v^{-1}]\langle\partial_v\rangle}

\newcommand{\loccit}{loc.\ cit.}

\begin{document}


\maketitle



\begin{prelims}


\def\abstractname{Abstract}
\abstract{We give a formula computing the irregular Hodge numbers for a confluent hypergeometric differential equation.}

\vspace{0.10cm}


\vspace{0.10cm}


\vspace{0.35cm}

\languagesection{Fran\c{c}ais}{%

\textbf{Titre. Nombres de Hodge irr\'eguliers des \'equations diff\'erentielles hyperg\'eom\'etriques confluentes} \commentskip \textbf{R\'esum\'e.} Nous donnons une formule calculant les nombres de Hodge irr\'eguliers pour les \'equations diff\'erentielles hyperg\'eom\'etriques confluentes.}

\end{prelims}


\newpage

\setcounter{tocdepth}{1} \tableofcontents

\section{Introduction}
Differential equations with irregular singularities occur in various branches of Algebraic geometry, like mirror symmetry or the theory of exponential periods. They are also of interest as providing a complex analogue of $\ell$-adic sheaves with wild ramification in positive characteristic. Irregular Hodge theory, as initiated by Deligne (\cf \cite{Deligne8406}), gives, for a large class of such equations, a convenient analogue to Hodge theory for Picard-Fuchs equations appearing more classically in complex Algebraic geometry. It is proved in \cite{Bibi15} that any rigid irreducible differential equation on the Riemann sphere, having regular singularities or not, and having real formal exponents at each singular point, underlies a variation of irregular Hodge structures away from its singular points. In this article, we consider the first and most classical example of such irregular differential equations, namely that of confluent hypergeometric differential equations, and we aim at determining the ranks of the irregular Hodge bundles. In the non-confluent case, the Hodge numbers of this variation, as well as the limiting Hodge numbers at the singularities, have been computed by R.\,Fedorov \cite{Fedorov15}, relying on~\cite{D-S12}.

Let $\alphag=(\alpha_1,\dots,\alpha_n)$ and $\betag=(\beta_1,\dots,\beta_m)$ be finite increasing sequences of length $n\geq0$ and $m\geq0$ (not both zero) of real numbers in $[0,1)$ (with the convention that a sequence of length zero is empty). We say that the pair $(\alphag,\betag)$ is non-resonant if $\alpha_i\neq\beta_j$ for any $i\in\{1,\dots,n\}$ and $j\in\{1,\dots,m\}$. All along this article, we make the following assumption.

\begin{assumption}\label{ass:nonresonant}
The pair $(\alphag,\betag)$ is non-resonant.
\end{assumption}

We consider the possibly confluent hypergeometric differential equation
\begin{equation}\label{eq:conflhypergeom}
\ccH=\ccH(\alphag,\betag)=\prod_{i=1}^n(t\partial_t-\alpha_i)-t\prod_{j=1}^m(t\partial_t-\beta_j),
\end{equation}
with the usual convention that a product indexed by the empty set is equal to~$1$. We know that the associated meromorphic flat bundle $\cH(\alphag,\betag)$ on $\PP^1$ is irreducible (\cf\cite[Cor.\,3.2.1]{Katz90}), that its index of rigidity is equal to $2$ (\cf\cite[Th.\,3.7.1\,\&\,Th.\,3.7.3]{Katz90}), and that it is rigid (\cf\cite[Th.\,4.7\,\&\,Th.\,4.10]{B-E04}). If $n=m$, it has singularities at $0,1,\infty$, and they are regular. If $n> m$, it has an irregular singularity at $t=\infty$, a regular singularity at $t=0$, and no other singularity. If $n<m$, the roles of $0$ and $\infty$ are exchanged.

Since the $\alpha_i$'s and $\beta_j$'s are real, the local monodromy of $\cH(\alphag,\betag)$ at its regular singularities is unitary and the formal monodromy at its irregular singular point is also unitary. If $n=m$ (regular singularities), there exists a variation of polarizable Hodge structure on $\CC^*$ that $\cH(\alphag,\betag)$ underlies, which is unique up to a shift of the Hodge filtration.

In this article, we consider the confluent case $n>m$ (the case $n<m$ can be obtained by a change of variable $t\mto1/t$), so we fix two integers $n>m\geq0$ and we set $\mu=n-m>0$. Then $\cH(\alphag,\betag)$ has a regular singularity at $t=0$, an irregular singularity of pure slope $1/\mu$ at $t=\infty$, and no other singularity. By \cite[Th.\,0.7]{Bibi15}, the minimal extension $\cH^{\min}(\alphag,\betag)$ at $t=0$ of $\cH(\alphag,\betag)$ underlies a unique pure object $\cT^{\min}(\alphag,\betag)$ of the category $\IrrMHM(\PP^1_t)$ of irregular mixed Hodge modules on~$\PP^1$, and it comes equipped with a irregular Hodge filtration. In contrast with the non-confluent case, this filtration is indexed by a set $A+\ZZ$, where $A$ is a finite set in~$\RR$, and this filtration is unique up to a~shift of $A$ by a real number. We determine these numbers and their multiplicities in term of the pair $(\alphag,\betag)$.

\begin{theoreme*}
The jumps of the irregular Hodge filtration $F^\cbbullet_\irr\cH(\alphag,\betag)$ occur (up to a~global $\RR$-shift) at
\begin{equation}\label{eq:hypergeomstar}
\rho(k):=\mu\alpha_k-k+\#\{i\mid\beta_i<\alpha_k\},
\end{equation}
and for any $p\in\RR$ we have
\begin{equation}\label{eq:hypergeomstarstar}
\rk\gr^p_{F_\irr}\cH(\alphag,\betag)=\#\rho^{-1}(p).
\end{equation}
\end{theoreme*}

\begin{remarques*}\mbox{}\label{rem:hypergeom}
\begin{enumeratei}
\item \label{rem:hypergeom1}
Recall that the irregular Hodge filtration is unique up to a shift by a real number, so the formula \eqref{eq:hypergeomstarstar} above is understood up to an $\RR$-shift of the filtration.
\item \label{rem:hypergeom2}
The statement and proof of the theorem hold under the assumption that \hbox{$n>m$}. However, the formula remains meaningful when $m=n$, and it gives back the formula of R.\,Fedorov \cite{Fedorov15} (if we notice that R.\,Fedorov considers the local system of solutions of~$\ccH$, while we consider the dual one of horizontal sections\footnote{This remark is due to Nicolas Martin, \cf \cite[Lem.\,3.6]{Martin18b}.}). Another proof is given in \cite[Th.\,3.4]{Martin18b}, where the vanishing cycle Hodge number at $t=1$ is also made explicit. Notice that our proof of the theorem relies on the previous result by R.\,Fedorov.

\item \label{rem:hypergeom3}
The case where $m=0$ is due to A.\,Casta\~{n}o\,Dom\'{\i}nguez and C.\,Sevenheck \cite[Th.\,4.7]{CD-S17}, and another proof in this case is given by both authors in \cite[\S3.2.c]{Bibi15}. Moreover, A.\,Casta\~{n}o\,Dom\'{\i}nguez, Th.\,Reichelt and C.\,Sevenheck have also obtained the case $n>m=1$ \cite[Th.\,5.8]{CD-R-S18} with different methods, relying on their results on GKZ systems.
\end{enumeratei}
\end{remarques*}

\section{Fourier transforms of Kummer pullbacks of hypergeometrics}\label{subsec:Katz}
We recall here a useful result of N.\,Katz (\cf\cite[Th.\,6.2.1]{Katz90}) which reduces the study of confluent hypergeometric differential equations to that of regular ones. Recall that we set $\mu:=n-m$. We denote by $\ov\betag$ the sequence of length $n$ obtained by concatenating and reordering the sequences $\betag$ and $0,1/\mu,\dots,(\mu-1)/\mu$. We also have
\[
0\leq\ov\beta_1\leq\cdots\leq\ov\beta_n<1.
\]

In the remaining part of this section, we will make the following assumptions, the first one being mainly for convenience.

\pagebreak[2]
\begin{assumption}\label{ass:stronglynonresonant}\mbox{}
\begin{enumeratei}
\item\label{ass:stronglynonresonant1}
We have $\alpha_i,\beta_j\in(0,1)$.
\item\label{ass:stronglynonresonant2}
The pair $(\alphag,\ov\betag)$ is non-resonant (in particular, $\mu\alpha_i\not\in\ZZ$ for any $i=1,\dots,n$).
\end{enumeratei}
\end{assumption}

The theorem of N.\,Katz relates the confluent $\cH(\alphag,\betag)$ to the non-confluent $\cH(\ov\betag,\alphag)$. We recall below this correspondence.

The confluent hypergeometric differential system $H:=\Clt/\Clt\cdot\ccH$ defined by \eqref{eq:conflhypergeom} is localized at $t=0$ because $\alpha_i\notin\ZZ$. We will also consider it as a $\Cltm$-module. We denote by $\rho_\mu:\CC^*_v\to\CC^*_t$ the cyclic covering $v\mto t=v^\mu$. Then $\rho_\mu^+H$ is a well-defined $\Clvm$-module, defined as such by the operator
\[
\ccH_\mu=\prod_{i=1}^n(\tfrac1\mu v\partial_v-\alpha_i)-v^\mu\prod_{j=1}^m(\tfrac1\mu v\partial_v-\beta_j).
\]
We set $H_\mu:=\Clv/(\ccH_\mu)$. Since $\mu\alpha_i\notin\ZZ$, we have \hbox{$H_\mu=H_\mu(!0)=H_\mu(*0)=H_\mu^{\min}$}. It also follows that~$H_\mu$ does not have any nonzero constant $\Clv$-submodule or quotient module.

Let $\wh H_\mu$ be the Fourier transform of $H_\mu$ by the Fourier correspondence $\tau=\partial_v$, $\partial_\tau=-v$. We have $\wh H_\mu=\Cltau/(\wh\ccH_\mu)$, with
\[
\wh\ccH_\mu=\prod_{i=1}^n(\tfrac1\mu \tau\partial_\tau+\alpha_i+\tfrac1\mu)-(\partial_\tau)^\mu\prod_{j=1}^m(\tfrac1\mu \tau\partial_\tau+\beta_j+\tfrac1\mu).
\]
As a consequence, $\wh H_\mu$ does not have any nonzero constant sub or quotient $\Cltau$-module, and it does not have either any sub or quotient $\Cltau$-module supported at the origin, \ie it is a minimal extension at the origin.

We set
\[
\alpha'_i=\alpha_i+1/\mu,\quad\beta'_j=\beta_j+1/\mu,
\]
With the change of variable $\iota:\CC^*\to\CC^*$ given by $\iota(\tau)=\tau'=1/\tau$, and the choice of the generator $\tau^{\prime-1}$ instead of $1$, $\iota^*(\wh H_\mu(*0))$ is defined by the operator
\begin{align*}
\ccH'_\mu&=\prod_{i=1}^n(\tfrac1\mu \tau'\partial_{\tau'}+1/\mu-\alpha'_i)-\mu^\mu\tau^{\prime\mu}\!\prod_{\ell=0}^{\mu-1}(\tfrac1\mu\tau'\partial_{\tau'}\!+\!1/\mu+\ell/\mu)\prod_{j=1}^m(\tfrac1\mu \tau'\partial_{\tau'}+1/\mu-\beta'_j)\\
&=\prod_{i=1}^n(\tfrac1\mu \tau'\partial_{\tau'}-\alpha_i)-\mu^\mu\tau^{\prime\mu}\prod_{\ell=1}^{\mu}(\tfrac1\mu\tau'\partial_{\tau'}+\ell/\mu)\prod_{j=1}^m(\tfrac1\mu \tau'\partial_{\tau'}-\beta_j)
\end{align*}
Clearly, we have $\wh H_\mu(*0)\simeq\iota^*\rho_\mu^*(H'(*0))$ with $H'(*0)=\CC[x,x^{-1}]\langle x\partial_x\rangle/(\ccH')$ and
\[
\ccH'=\prod_{i=1}^n(x\partial_x-\alpha_i)-\mu^\mu x\prod_{\ell=1}^{\mu}(x\partial_x+\ell/\mu)\prod_{j=1}^m(x\partial_x-\beta'_j).
\]
Let us consider $\ccH''$ obtained by reducing the exponents $-\ell/\mu$ modulo $\ZZ$. We obtain the sequences of exponents $\alpha_1,\dots,\alpha_n$, $\beta_1,\dots,\beta_m$ and $(1-\ell/\mu)_{\ell=1,\dots,\mu}$, the latter two being grouped as $\ov\beta_1,\dots,\ov\beta_n$. Hence,
\begin{equation}\label{eq:ccHprime}
\ccH''=\prod_{i=1}^n(x\partial_x-\alpha_i)-\mu^\mu x\prod_{i=1}^n(x\partial_x-\ov\beta_i).
\end{equation}
We set $H''(*0)=\CC[x,x^{-1}]\langle x\partial_x\rangle/(\ccH'')\simeq H'(*0)$. We have obtained that
\[
\wh H_\mu(*0)\simeq\iota^*\rho_\mu^*(H''(*0)).
\]
In conclusion, $H=H(*0)=\cH(\alphag,\betag)$ is obtained from $H''(*0)\simeq\cH(\ov\betag,\alphag)$ by the following operations:\footnote{We neglect here to take into account the coefficient in front of the right-hand term, that is, we consider hypergeometrics up to isomonodromy deformations.}
\begin{enumeratea}
\item
Kummer pullback $\rho_\mu:\tau'\mto x=\tau^{\prime\mu}$,
\item
change of variables $\tau'=1/\tau$,
\item
Fourier transform $v=-\partial_\tau$, $\partial_v=\tau$,
\item
Kummer descent by $\rho_\mu:v\mto t=v^\mu$.
\end{enumeratea}

\section{Reduction of the proof of the theorem to the case where Assumption~\ref{ass:stronglynonresonant} is fulfilled}

Recall that we assume that the pair $(\alphag,\betag)$ is non-resonant (Assumption \ref{ass:nonresonant}). Then for $\gamma>0$ small enough, setting $\alpha''_i=\gamma+\alpha_i$ and $\beta''_j=\gamma+\beta_j$, the sequences $\alphag'',\betag''$ are in $(0,1)$, increasing, and remains non-resonant. Moreover, one can choose $\gamma$ such that the pair $(\alphag'',\ov{\betag''})$ is non-resonant, \ie Assumption \ref{ass:stronglynonresonant} is fulfilled for it. Moreover, since the irregular Hodge filtration is defined up to an $\RR$-shift, we can add $\mu\gamma$ to Formula~\eqref{eq:hypergeomstar}. In~order to reduce the proof of the theorem to the case where Assumption \ref{ass:stronglynonresonant} is fulfilled, we~apply the following general lemma to the case of the rank-one local system $\ccL$ on~$\CC^*$ with monodromy $\exp(-\twopii\gamma)$ around $0$.

\begin{lemme}
Let $\ccM$ be a rigid irreducible holonomic $\cD_{\PP^1}$-module and let $\ccL$ be a rank-one meromorphic flat bundle on $\PP^1$ with poles along $\Sigma\subset\PP^1$, which is locally formally unitary. Let $\ccM'$ be the image of the map $\Gamma_{[!\Sigma]}(\ccM\otimes\ccL)\to\Gamma_{[*\Sigma]}(\ccM\otimes\ccL)$, which is also rigid irreducible holonomic. Then the jumping indices and ranks of the irregular Hodge filtration for $\ccM$ and $\ccM'$ are the same, up to an $\RR$-shift of the indices.
\end{lemme}

\begin{proof}
We use, in the statement and the proof of the lemma, the notation as in the proof of \cite[Prop.\,2.69]{Bibi15}. We can write $\ccL=(\cO_{\PP^1}(*\Sigma),\rd+\rd\psi+\omega)$, where $\Sigma$ is the pole divisor of $\ccL$, $\psi$ is a global section of $\cO_{\PP^1}(*\Sigma)$ and~$\omega$ is a one-form with at most simple poles at $\Sigma$ (in the application to hypergeometrics we have in mind, we have $\psi=0$). Moreover, $\ccL$ is locally formally unitary if and only if $\ccL_\regg:=(\cO_{\PP^1}(*\Sigma),\rd+\omega)$ is unitary, \ie the residues of $\omega$ at $\Sigma$ are real.

It is shown in \loccit\ that there exists
\begin{itemize}
\item
a proper morphism $\map:X\to\PP^1$, with $X$ smooth projective,
\item
a normal crossing divisor $D$ in $X$ and a subdivisor $H\subset D$,
\item
a regular holonomic $\cD_X$-module $\ccN$ underlying a mixed Hodge module,
\item
a meromorphic function $\mero$ with poles in $H$,
\end{itemize}
such that $\ccM$ is the image of
\[
\map^0_\dag(\ccE^\mero\otimes\Gamma_{[!H]}\ccN)\to \map^0_\dag(\ccE^\mero\otimes\Gamma_{[*H]}\ccN).
\]
Set $D_1=\map^{-1}(\Sigma)$. By a suitable change of data as above, we can assume that the pole and zero divisors of~\hbox{$\mero+\psi\circ\map$} do not intersect, and that $D\cup D_1$ is a normal crossing divisor. Then (\cf\loccit) $\ccM'$ is obtained as the image of
\[
\map^0_\dag\bigl(\ccE^{\mero'}\otimes\Gamma_{[!H']}(\map^+\ccL_\regg\otimes\ccN)\bigr)\to \map^0_\dag\bigl(\ccE^{\mero'}\otimes\Gamma_{[*H']}(\map^+\ccL_\regg\otimes\ccN)\bigr).
\]
The irregular Hodge filtration on $\ccM$ is obtained by pushing forward by $f^0_\dag$ the irregular Hodge filtrations on $\ccE^\mero\otimes\Gamma_{[!H]}\ccN$ and $\ccE^\mero\otimes\Gamma_{[*H]}\ccN$, and by considering the image of it by the morphism above, and similarly for $\ccM'$. We notice that, away from $\Sigma$, both constructions possibly differ only because the irregular Hodge filtrations on $\ccN$ and $\map^+\ccL_\regg\otimes\ccN$ possibly differ. Since the only choice involved is that of the jumping index of the irregular Hodge filtration of $\ccL_\regg$, which can be an arbitrary real number, they actually do not differ, and we obtain the desired result.
\hfill $\Box$
\end{proof}

\section{Nearby cycles for the Kummer pullback}\label{subsec:nearby}
We take up here the notation as in \cite{D-S12}. Let $(V,F^\cbbullet V,\nabla)$ be a filtered flat vector bundle on the punctured disc $\Delta_x^*$ with coordinate $x$ underlying a variation of polarized complex Hodge structure. We denote by~$V^a$ the Deligne extension of $V$ on $\Delta_x$ on which the residue of $\nabla$ has eigenvalues in $[a,a+1)$, and we set $V^{-\infty}=\bigcup_aV^a$, which is a free $\cO_\Delta(*0)$-module of finite rank. For any $p\in\ZZ$ we set $F^pV^a=j_*F^pV\cap V^a$, where $j:\Delta^*\hto\Delta$ denotes the inclusion. This is a locally free $\cO_\Delta$-module and multiplication by $x$ induces an isomorphism $F^pV^a\isom F^pV^{a+1}$, so that for fixed $p$ and $a$, $\gr^p_F\gr^{a+k}(V):=F^pV^{a+k}/(F^{p+1}V^{a+k}+F^pV^{>a+k})$ has dimension independent of $k\in\ZZ$. For $\chi=\exp-\twopii a$ with $a\in\RR$, we set
\[
\nu_\chi^p(V)=\dim\gr^p_F\gr^a(V).
\]

For $\mu\in\NN^*$,\,let $\rho:\Delta_y\to\Delta_x$\,be the cyclic ramification $y\to x=y^\mu$\,of order $\mu$.\,The data $(\rho^*V,\rho^*F^\cbbullet V,\rho^*\nabla)$ underlies a variation of polarized complex Hodge structure.

\begin{lemme}\label{lem:KummerHodge}
We have
\[
\nu_\lambda^p(\rho^*V)=\sum_{\chi\mid\chi^{^\mu}=\lambda}\nu_\chi^p(V).
\]
\end{lemme}

\begin{proof}
Considering the germs at the origin, we have a natural identification
\[
\rho^*V^{-\infty}=\bigoplus_{j=0}^{\mu-1}y^j\otimes V^{-\infty}
\]
with a natural structure on the right-hand side, which leads to
\[
(\rho^*V)^b=\bigoplus_{j=0}^{\mu-1}y^j\otimes V^{(b-j)/\mu}\quad(b\in\RR),
\]
and similarly
\[
F^p(\rho^*V)^b:=(j_*\rho^*F^p)\cap(\rho^*V)^b=\bigoplus_{j=0}^{\mu-1}y^j\otimes F^pV^{(b-j)/\mu}.
\]
The lemma follows.
\hfill $\Box$
\end{proof}

Let us apply this formula to $H''(*0)$ defined by \eqref{eq:ccHprime} and its pullback $\iota^*\wh H_\mu(*0)$. Assump\-tion \ref{ass:stronglynonresonant} is supposed to hold. We consider nearby cycles at $x=0$ and $\tau'=\nobreak0$. We set $\chi_k=\exp(-\twopii\alpha_k)$ and $\lambda_k=\exp(-\twopii\mu\alpha_k)$. The formula of \cite[Th.\,3(b)]{Fedorov15} reads, since the nilpotent part for each eigenvalue of the monodromy of $H''(*0)$ at $x=0$ consists of one Jordan block,
\[
\nu_{0,\chi_{_k}}^p(\ccH'')=
\begin{cases}
1&\text{if }p=p_k:=\#\{i\mid\ov\beta_i<\alpha_k\}-k,\\
0&\text{otherwise}.
\end{cases}
\]
Note that
\[
\#\{j\in\{0,\dots,\mu-1\}\mid j/\mu<\alpha_k\}=[\mu\alpha_k],
\]
so that
\[
p_k=[\mu\alpha_k]+\#\{i\mid\beta_i<\alpha_k\}-k.
\]

Lemma \ref{lem:KummerHodge} gives:
\[
\nu_{0,\lambda_k}^p(\iota^*\wh H_\mu(*0))=
\begin{cases}
\#\{j\mid\mu\alpha_j\equiv\mu\alpha_k\bmod\ZZ\text{ and }p_j=p_k\}&\text{if }p=p_k,\\
0&\text{otherwise}.
\end{cases}
\]
Let us denote by $\{\mu\alpha_j\}\in[0,1)$ the fractional part of $\mu\alpha_j$ (it belongs to $(0,1)$, due to Assumption \ref{ass:stronglynonresonant}). Since $\rho(j)=\{\mu\alpha_j\}+p_j$, the previous formula can be rewritten as
\[
\nu_{0,\lambda_k}^p(\iota^*\wh H_\mu(*0))=
\begin{cases}
\#\{j\mid\rho(j)=\rho(k)\}&\text{if }p=p_k,\\
0&\text{otherwise},
\end{cases}
\]
and, after applying $\iota$, it reads
\begin{equation}\label{eq:nupinfHprimemu}
\nu_{\infty,\lambda_k}^p(\wh H_\mu(*0))=
\begin{cases}
\#\{j\mid\rho(j)=\rho(k)\}&\text{if }p=p_k,\\
0&\text{otherwise}.
\end{cases}
\end{equation}

\section{The filtered inverse stationary phase formula and irregular Hodge \hbox{numbers}}\label{subsec:Laplace}
By \cite[Prop.\,2.61]{Bibi15}, the general fibre of the Laplace transform of a mixed Hodge module on $\Afu$ carries a canonical irregular Hodge structure. We will give a formula for the irregular Hodge numbers in terms of the limit mixed Hodge structure at infinity of the mixed Hodge module.

We start by recalling some of the results in \cite{Bibi08}, since the way we formulate them is implicit in loc.~cit.

Let $(M,F^\cbbullet M)$ be a well-filtered regular holonomic $\Cltau$-module underlying a polarizable pure complex Hodge module. For the sake of simplicity, and since we will only use the result in this setting, we assume that the monodromy of $M$ around $\tau=\infty$ does not have $1$ as an eigenvalue. We associate with $(M,F^\cbbullet M)$
\begin{itemize}
\item
the Rees module $R_FM$,
\item
the localized Laplace transform $G$, that we regard as a $\CC[v,v^{-1}]$-module, and which is free of finite rank as such,
\item
the Brieskorn lattice $G^0=G_{(F)}^0$ associated to the filtration, which is a free $\CC[u]$-module ($u=v^{-1}$) with an action of $u^2\partial_u$ (\cf\eg\cite[App.]{S-Y14}),
\item
the Rees module $R_{G_{(F)}}(G)$ attached to the decreasing filtration $G_{(F)}^p=u^pG^0$.
\end{itemize}

Let $\ccM$ be the $\cD_{\PP^1}$-module such that $\ccM=\ccM(*\infty)$ and $M=\Gamma(\PP^1,\ccM)$. Our assumption above implies that $\ccM$ is equal to its minimal extension at $\tau=\infty$. We denote by $\tau'=1/\tau$ the coordinate centered at $\infty$ and by $V^\cbbullet\ccM$ the $V$-filtration of $\ccM$ with respect to $\tau'$. For $\alpha\in\RR$ and $\lambda=\exp(-\twopii\alpha)$, we set $\psi_{\tau',\lambda}\ccM=\gr^\alpha_V\ccM$.

The Hodge filtration $F^\cbbullet M$ extends naturally to $V^\alpha\ccM$, as indicated in Section \ref{subsec:nearby}. In such a way, $(\ccM,F^\cbbullet\ccM)$ is strictly specializable at $\tau'=0$. The space $\psi_{\tau',\lambda}\ccM$ comes equipped with the induced filtration given by $F^\cbbullet\psi_{\tau',\lambda}\ccM=F^\cbbullet\gr^\alpha_V\ccM$, from which we construct the Rees module $R_F\psi_{\tau',\lambda}\ccM$.

On the other hand, let $V^\cbbullet G$ be the $V$-filtration of $G$ with respect to the function~$v$. We set similarly $\psi_{v,\lambda}G=\gr^\alpha_VG$, and the filtration $G_{(F)}^\cbbullet$ induces on it the filtration $G_{(F)}^\cbbullet\psi_{v,\lambda}G$, form which we construct the Rees module $R_{G_{(F)}}\psi_{v,\lambda}G$.

The ``inverse stationary phase formula'' of \cite[Prop.\,4.1(iv)]{Bibi05b} applied to $\cM=R_F\ccM$, together with \cite[Lem.\,5.20\,$(*)_\infty$]{Bibi08} give, for $\lambda\neq1$,
\begin{equation}\label{eq:GF}
R_{G_{(F)}}\psi_{v,\lambda}G\simeq R_F\psi_{\tau',\lambda}\ccM,
\end{equation}
that we can regard in each degree $p$ as an isomorphism $G_{(F)}^p\psi_{v,\lambda}G\simeq F^p\psi_{\tau',\lambda}\ccM$.

On the other hand, let us set $G_{|u=1}=G^0/(u-1)G^0$. The irregular Hodge filtration is the filtration induced on $G_{|u=1}$ by the $V$-filtration of $G$ with respect to the coordinate $v=1/u$ (\cf\cite[Def.\,3.2]{Bibi15}). In the decreasing version, we have
\[
F^\gamma_\irr G_{|u=1}=(V^\gamma\cap G^0)/(V^\gamma\cap G^0)\cap (u-1)G^0.
\]
By \cite[(1.3)]{Bibi08} we have (by replacing there $z$ with $v$ and taking $z_o=0$)
\[
\dim\gr^\gamma_{F_\irr}G_{|u=1}=\dim\bigl((V^\gamma\cap G^0)/[(V^{>\gamma}\cap G^0)+(V^\gamma\cap G^1)]\bigr).
\]
Let us set $\alpha=\{\gamma\}:=\gamma-[\gamma]\in[0,1)$, so that $v^{-[\gamma]}(V^\gamma\cap G^0)=V^\alpha\cap G^{[\gamma]}$. Then, for $p\in\ZZ$ and $\alpha\in[0,1)$, we find\vspace*{-3pt}
\[
\dim\gr^{\alpha+p}_{F_\irr}G_{|u=1}=\dim\gr^p_{G_{(F)}}\psi_{v,\lambda}G.
\]
Together with \eqref{eq:GF} we obtain:
\begin{equation}\label{eq:FirrFourier}
\dim\gr^{\alpha+p}_{F_\irr}G_{|u=1}=\dim\gr^p_F\psi_{\tau',\lambda}\ccM\quad(\alpha\in[0,1),\;\lambda=\exp(-\twopii\alpha)).
\end{equation}

The previous relation was proved by means of the results of \cite{Bibi05b,Bibi08}, provided $(M,F^\cbbullet M)$ underlies a pure polarizable Hodge module. Let us remark that it also holds if $(M,F^\cbbullet M)$ underlies a mixed Hodge module: this is proved by induction on the weight by considering short exact sequences $0\to W_{k-1}\to W_k\to\gr_k^W\to0$. Indeed, these exact sequences for mixed Hodge module are $F$-strict, and their Laplace transforms are $F_\irr$\nobreakdash-strict, because they underlie exact sequences of irregular mixed Hodge modules.

\section{End of the proof of the theorem}
We use the notation of Section \ref{subsec:Katz}. We first argue as in \cite[\S3.2.c]{Bibi15} to relate the irregular Hodge filtration of $H$ and that of $H_\mu$. Moreover, we have already seen that we can reduce the proof of the theorem to the case where Assumption \eqref{ass:stronglynonresonant} holds, so that in particular $\mu\alpha_k\notin\ZZ$ and thus $\alpha_k\notin\ZZ$.

By \cite[Th.\,0.7]{Bibi15}, the minimal extension $\cH^{\min}(\alphag,\betag)$ at $t=0$ underlies a unique object $\cT^{\min}(\alphag,\betag)$ of $\IrrMHM(\PP^1_t)$, and it comes equipped with a unique (up to an $\RR$-shift) irregular Hodge filtration. We also denote by $\cT^{\min}(\alphag,\betag)$ the associated pure polarizable twistor $\cD$-module on $\PP^1_t$ and by $\cT(\alphag,\betag)$ the localized object in $\MTM^\intt(\PP^1_t,[*0])$. With Assumption \eqref{ass:stronglynonresonant}, we have $\cT^{\min}(\alphag,\betag)=\cT(\alphag,\betag)$.

The pullback $\cT_\mu(\alphag, \betag)$ of $\cT(\alphag, \betag)$ endows $\cH_\mu(\alphag, \betag)$ with the structure of an integrable mixed twistor $\cD$-module on $\PP^1_v$ localized at $v=0$, and its minimal extension $\cT_\mu^{\min}(\alphag, \betag)$ at $v=0$ is a polarizable twistor $\cD$-module which is pure, with associated $\CC[v]\langle\partial_v\rangle$-module $\cH_\mu^{\min}(\alphag, \betag)$. With Assumption \eqref{ass:stronglynonresonant}, we have $\cT_\mu^{\min}(\alphag,\betag)=\cT_\mu(\alphag,\betag)$. Since the covering $\rho_\mu$ is a smooth morphism, the rank and jumping indices of the irregular Hodge filtrations are not altered by the pullback by $\rho_\mu$, hence those of $\cH(\alphag,\betag)$ coincide with those of $\cH_\mu(\alphag,\betag)$.

Moreover, as a polarizable pure twistor $\cD$-module, $\cT_\mu^{\min}(\alphag,\betag)$ is obtained as the Fourier-Laplace transform, in the twistor sense, of the polarizable variation of Hodge structure that $\wh H_\mu$ underlies. Therefore, we can apply to it \cite[Prop.\,2.61]{Bibi15}, and also the results of Section \ref{subsec:Laplace}.

We apply Formula \eqref{eq:FirrFourier} with $M=\wh H_\mu$ and $G=H_\mu$, both defined in \S\ref{subsec:Katz}. Since the right-hand term concerns the behaviour at $\tau=\infty$, we can replace $\wh H_\mu$ with $\wh H_\mu(*0)$. We note that the assumption used in \eqref{eq:FirrFourier} is satisfied here, since the eigenvalues $\lambda_k=\exp(-\twopii\mu\alpha_k)$ are not equal to one, as a consequence of Assumption \eqref{ass:stronglynonresonant}.

Formula \eqref{eq:nupinfHprimemu} then gives\vspace*{-3pt}\enlargethispage{.7\baselineskip}%
$$
\dim\gr^{\{\mu\alpha_k\}+p}_{F_\irr}H=\dim\gr^{\{\mu\alpha_k\}+p}_{F_\irr}H_\mu=
\begin{cases}
\#\{j\mid\rho(j)=\rho(k)\}&\text{if }p=p_k,\\
0&\text{otherwise},
\end{cases}
$$
which is equivalent to \eqref{eq:hypergeomstarstar}.
\hfill $\Box$

\providecommand{\bysame}{\leavevmode\hbox to3em{\hrulefill}\thinspace}
\providecommand{\Zbl}[1]{\href{https://zbmath.org/#1}{Zbl-#1}}
%
%

\bibliographystyle{amsalpha}
\bibliographymark{References}
\def\cprime{$'$}

\end{document}